\documentclass[10pt,a4paper]{amsart}
\usepackage{amsmath,amsfonts,amssymb,amsthm,url,hyperref}
\usepackage{mathrsfs}
\usepackage[all,cmtip]{xy}
\usepackage{xcolor}

        \DeclareFontFamily{OT1}{pzc}{}
\DeclareFontShape{OT1}{pzc}{m}{it}{<-> s * [1.10] pzcmi7t}{}
\DeclareMathAlphabet{\mathpzc}{OT1}{pzc}{m}{it}

\newtheorem{theorem}{Theorem}[section]
\newtheorem{lemma}[theorem]{Lemma}
\newtheorem{corollary}[theorem]{Corollary}
\newtheorem{proposition}[theorem]{Proposition}

\newtheorem{observation}[theorem]{Observation}

\newtheorem{fact}[theorem]{Fact}

\theoremstyle{definition}
\newtheorem{example}[theorem]{Example}
\newtheorem{definition}[theorem]{Definition}
\newtheorem{remark}[theorem]{Remark}

\newcommand{\Z}{{\mathbb Z}}
\newcommand{\N}{{\mathbb N}}
\newcommand{\R}{{\mathbb R}}
\DeclareMathOperator{\supp}{supp}
\def\span{\operatorname{span}}
\newcommand{\lsc}{lsc}
\newcommand{\free}[1]{\ensuremath{\mathcal{F}({#1})}}
\def\Lip{\operatorname{Lip}}
\newcommand{\Lipz}[1]{\ensuremath{\Lip_0({#1})}}

\renewcommand{\AE}{$\mathcal{M}\hspace{-1.5pt}\mathpzc{ol}$}

\def\norm#1{\left|\left|#1\right|\right|}

\title{On free bases of Banach spaces
}

\author{Eva Perneck\'{a}}
\author{Jan Sp\v{e}v\'{a}k}

\begin{document}

\begin{abstract}
    We call a closed subset $M$ of a Banach space $X$ a {\em free basis} of $X$ if it contains the null vector and every Lipschitz map from $M$ to a Banach space $Y$, which preserves the null vectors can be uniquely extended to a bounded linear map from $X$ to $Y$. We then say that two complete metric spaces $M$ and $N$ are {\em \AE-equivalent} if they admit bi-Lipschitz copies $M'$ and $N'$, respectively that are free bases of a common Banach space satisfying $\span M'=\span N'$. 
    
    In this note, we compare \AE-equivalence with some other natural equivalences on the class of complete metric spaces. The main result states that \AE-equivalent spaces must have the same \v{C}ech-Lebesgue covering dimension. In combination with the work of Godard, this implies that two complete metric spaces with isomorphic Lipschitz-free spaces need not be \AE-equivalent. Also, there exist non-homeomorphic \AE-equivalent metric spaces, and, in contrast with the covering dimension, the metric Assouad dimension is not preserved by \AE-equivalence.
  
\end{abstract}



\maketitle
\section{Introduction}
Vector spaces throughout this paper are considered over the field $\R$ of real numbers. The zero vector of a vector space $X$ is denoted by $0_X$.
We say that two normed spaces are {\em isomorphic} provided that there exists a linear bijection between them which is simultaneously a homeomorphism.  Similarly, we call two metric spaces {\em bi-Lipschitz isomorphic} if there exists a bijection between them that is Lipschitz and whose inverse is also Lipschitz. By covering dimension of a metric space we mean the \v{C}ech-Lebesgue covering dimension (or equivalently the large inductive dimension) of its underlining topological space.

We will now introduce two notions that will be central to our work.
\begin{definition}
Given a Banach space $X$, we say that its closed subset $M$ is a 
{\em free basis} of $X$ provided that $0_X\in {M}$ and 
for every Banach space $Y$ and every Lipschitz map $f:M\to Y$ which preserves zero vectors it is possible to uniquely extend $f$ to a bounded linear map $\hat{f}:X\to Y$. 
\end{definition}

We note that every complete metric space is a free basis of some Banach space. This is thanks to the existence of free Banach spaces recalled at the beginning of Section \ref{section:basics} (see Remark \ref{rem:F-eqv:Free:isom} (i)).

\begin{definition}
Given two complete metric spaces $M$ and $N$, we say that $M$ and $N$ are {\em \AE-equivalent} provided that there are free bases $M'$ and $N'$ of common Banach space such that $M$ is bi-Lipschitz isomorphic to $M'$ and $N$ to $N'$, and that $\span{M'}=\span{N'}$.
\end{definition}

The notation of \AE-equivalence emphasizes the fact that $M$ and $N$ are \AE-equivalent if and only if their corresponding normed spaces of {\em molecules} constructed by Arens and Eells in \cite{AE} are isomorphic. This follows from the universal property of the completion of the space of molecules established by Weaver in \cite[Theorem 3.6]{W2}. The Banach space obtained by such completion will be referred to as {\em Lipschitz-free space} and described in Section \ref{section:basics}.

We will study the relations of \AE-equivalence to some other natural equivalences on the class of complete metric spaces. More specifically, we will deal with equivalences given by: bi-Lipschitz isomorphisms, homeomorphisms, covering dimension, and isomorphisms of the corresponding Lipschitz-free spaces (which we will refer to as {\em $\mathcal{F}$-equivalence}).

Obviously, a bi-Lipschitz isomorphism is a homeomorphism, and homeomorphisms preserve covering dimension. Also, bi-Lipschitz isomorphic spaces are \AE-equivalent and, because Lipschitz-free spaces are completions of spaces of molecules, \AE-equivalence implies $\mathcal{F}$-equivalence.
It is proved in \cite[Theorem 3.55]{W2} that if two complete metric spaces that are uniformly concave (see \cite[Definition 3.33]{W2}) have linearly isometric Lipschitz-free spaces, then there exists a bijection between such metric spaces and a positive constant $r$ such that the distance of any two images is just the $r$ multiple of the distance of their preimages. In particular, the metric spaces are bi-Lipschitz isomorphic and thus \AE-equivalent. However, as we will see in the sequel, mere $\mathcal{F}$-equivalence does not ensure \AE-equivalence, and the latter does not even imply that the metric spaces are homeomorphic. 

No other implication can be added to the following diagram (here {\em dim} stands for the relation of having the same covering dimension):

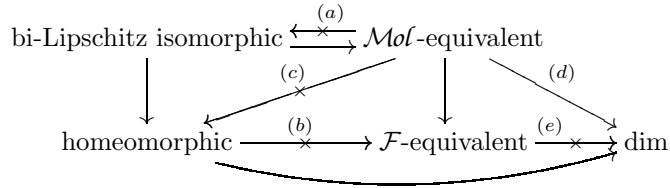
\begin{center} \medskip\hspace{1em}\xymatrix{\mbox{bi-Lipschitz isomorphic} \ar@{->}@<-.71ex>[r] \ar@{->}[d]
  &\mbox{\AE-equivalent}  \ar@{->}@<-.71ex>[d] \ar@{->}@<0ex>[ld]|-\times  \ar@{->}@<0ex>[ld]_{(c)} \ar@{->}@<-.71ex>[l]|-\times \ar@{->}@<-.71ex>[l]_{\hspace{23 pt}(a)}  \ar@{->}@<0ex>[rd]^{(d)}  \\
  \mbox{homeomorphic} \ar@{->}@/_1.3pc/[rr]     \ar@{->}@<0ex>[r]|-\times \ar@{->}@<0ex>[r]^{(b)} & \mathcal{F}\mbox{-equivalent} \ar@{->}@<-.ex>[r]|-\times \ar@{->}@<-.ex>[r]^{(e)}   &\mbox{dim}     
  }
  
\end{center}
\bigskip
\bigskip

The fact that the implication (a) does not hold is implicitly contained in \cite{AACD2}, it is discussed in detail in Remark \ref{rem:AE:not:bilip} and follows also from the non-existence of the implication (c). Non-trivial examples of homeomorphic spaces that are not $\mathcal{F}$-equivalent, and hence witness that the implication (b) does not hold, were constructed in \cite{NS,Godard}.  Naor and Schechtman showed in \cite{NS}  that the homeomorphic discrete spaces $\Z$ and $\Z^2$ with their Euclidean metrics do not have isomorphic Lipschitz-free spaces, while Godard proved in \cite{Godard} that the homeomorphic Cantor sets one of zero and the other of non-zero Lebesgue measure have non-isomorphic Lipschitz-free spaces. We will address the failure of the implication (c) in Section \ref{section:examples}, where we recall a construction of \AE-equivalent spaces that can be found for instance in \cite{AACD_jfa} (see also \cite{Kauf}) and leads to examples of non-homeomorphic \AE-equivalent spaces such as in Example \ref{ex:AE:not:homeo}. Note that there are many examples of non-homeomorphic $\mathcal{F}$-equivalent spaces. To mention a few, see \cite{AACD1,G-LP,Godard,Kauf,W2}. Nevertheless,  these results are based on (techniques similar to) the Pe{\l}czy\'nski's decomposition method and therefore it is not clear, whether they can give examples of \AE-equivalent spaces. We also point out on two classical spaces that \AE-equivalence does not preserve Assouad dimension (Example \ref{ex:assouad:not:preserved:AE}). This should be compared with Corollary \ref{thm:same:dim:for:fin:supported} of the last section. 

In the main result of our note, Theorem \ref{them:main} and Corollary \ref{thm:same:dim:for:fin:supported}, we prove that \AE-equivalent spaces must have the same covering dimension, i.e. the implication (d). On the other hand, by a result due to Godard, \cite[Corollary 3.4]{Godard}, the closed unit interval and a Cantor set of a positive measure have isomorphic Lipschitz-free spaces and are therefore $\mathcal{F}$-equivalent while having different covering dimensions. This shows the non-existence of the implication (e). So in particular, $\mathcal{F}$-equivalence does not imply \AE-equivalence. Let us also mention in this context a classical open problem in the area of Lipschitz-free spaces, which is to determine whether the Lipschitz-free spaces over Euclidean spaces of different dimensions strictly greater than one might be isomorphic. Corollary \ref{thm:same:dim:for:fin:supported} implies that, should there exist an isomorphism, it cannot preserve the finiteness of supports introduced in \cite[Definition 2.5]{APPP}.

\section{Some basic properties of free bases and {\texorpdfstring{\AE}--equivalence}}
\label{section:basics}

A pointed metric space is a metric space $M$ with a distinguished base point $e_M\in M$. We consider a base point of a normed space to be its null vector. 
Given a Banach space $Y$, we denote by $\Lipz{M,Y}$ the Banach space of all $Y$-valued Lipschitz maps that preserve the base points, endowed with the norm which assigns to each $f\in\Lipz{M,Y}$ its Lipschitz number $L(f)$. As usually, we write $\Lipz{M}$ instead of $\Lipz{M,\R}$.

Recall that given a pointed metric space $M$, the Lipschitz-free space $\free{M}$ is the unique (up to a linear isometry) Banach space characterized by the following universal property (see \cite[Theorem 3.6]{W2}):
There is an isometry $\delta: M\to\free{M}$ which preserves the base points and for every Banach space $Y$ and every $f\in\Lipz{M,Y}$, there is a unique bounded linear map $\bar{f}:\free{M}\to Y$ making the following diagram commutative

\begin{equation}\label{eq:free:diagram}
     \medskip\hspace{1em}
    \xymatrix{M  \ar@{->}[rd]^f \ar@{->}[d]_{\delta} & \\
               \free{M} \ar@{->}[r]_{\bar{f}} & Y,}
\end{equation}
and such that $||{\bar{f}}||=L(f)$, where $L(f)$ is the Lipschitz number of $f$. We will write $\delta_M$ instead of $\delta$ when we need to emphasize the underlying metric space $M$. Moreover, for $Y=\R$ the map that assigns to each $f$ the bounded linear functional $\bar{f}$ as in \eqref{eq:free:diagram} is a linear isometry from $\Lipz{M}$ onto $\free{M}^*$ (see for instance \cite[Theorem 3.3]{W2}). The weak$^\ast$ topology of $\Lipz{M}$ induced by $\free{M}$ agrees on bounded sets with the topology of point-wise convergence of functions.
It is a simple and folklore fact, that $\free{M}$ does not depend on the choice of a base point of $M$.

Now we relate the notions of free bases and \AE-equivalence to Lipschitz-free spaces. It is immediate in view of the extension property of the free bases and the universal property \eqref{eq:free:diagram} of the Lipschitz-free spaces.
\begin{observation}\label{ob:M:free:basis:of:F(M)}
    \begin{itemize}
   \item[(i)] If $M$ is a free basis of a Banach space $X$, then $X$ is isomorphic to $\free{M}$.
   \item[(ii)] Two complete metric spaces $M,N$ are \AE-equivalent if and only if  there exists an isomorphism between $\free{M}$ and $\free{N}$ that maps $\span \delta_M(M)$ onto $\span \delta_N(N)$. 
   \end{itemize}
\end{observation}

In the following remark, we summarize some basic properties of free bases and \AE-equivalent spaces that follow from the definitions and Observation \ref{ob:M:free:basis:of:F(M)}.

\begin{remark}\label{rem:F-eqv:Free:isom}
\begin{itemize}
   \item[(i)]  For a metric space $M$, $\delta(M)$ is a free basis of $\free{M}$. Since $\free{M}$ exists for every metric space $M$, it follows that every complete metric space can be isometrically embedded into some Banach space as its free basis.

   \item[(ii)]  Complete metric spaces $M,N$ are $\mathcal{F}$-equivalent if and only if they admit bi-Lipschitz copies $M',N'$, respectively, that are free bases of a common Banach space.

    \item[(iii)]  By \cite[Corollary 3.46]{W2}, infinite-dimensional reflexive Banach spaces do not admit free bases.

    \item[(iv)] Note that if $M$ is a free basis of $X$, then there is $K>0$ such that for every Lipschitz map $f$ from $M$ to a Banach space $Y$ that preserves the base point, its bounded linear extension $\hat{f}:X\to Y$ satisfies  $||\hat{f}||\le K\cdot L(f)$.     Indeed, the extension $\hat{\delta}:X\to\free{M}$ of the isometric embedding $\delta: M\to\free{M}$ is an isomorphism of  $X$ and $\free{M}$. From \eqref{eq:free:diagram} we get $\hat{f}=\bar{f}\circ\hat{\delta}$. Since ${||\bar{f}||}=L(f)$, it suffices to put $K=||{\hat{\delta}}||$. Thus the definition of a free basis can be expressed formally stronger by requiring the existence of such a $K>0$.

    \item[(v)] If $M$ is a free basis of a Banach space $X$, then $M\setminus\{0_X\}$ must be linearly independent and $\span{M}$ is dense in $X$. Indeed, the first property follows from the existence of linear extensions of Lipschitz maps from $M$ to the whole $X$ (or just $\span M$), and the second follows from its uniqueness, both required by the definition of the free basis.
\end{itemize}
\end{remark}

We proceed by characterizing the \AE-equivalence of metric spaces in terms of the corresponding spaces of Lipschitz functions.

\begin{lemma}\label{lemma:span:in:span:iff:point:cont}
Let $M$ and $N$ be metric spaces such that there exists a linear isomorphism $T$ from $\free{M}$ onto $\free{N}$. Then $T(\delta_M(M))\subseteq \span(\delta_N(N))$ if and only if the adjoint $T^*:\Lipz{N}\to\Lipz{M}$ is continuous in the topology of point-wise convergence.
\end{lemma}

\begin{proof}
First assume that $T(\delta_M(M))\subseteq \span(\delta_N(N))$ and consider a net $(f_i)\subseteq \Lip_0(N)$ (not necessarily bounded) that converges point-wise to some $f\in\Lip_0(N)$. Fix $x\in M$. Then there exist $k\in\N$, $y_1,\ldots, y_k\in N$, $a_1,\ldots, a_k\in\R$ such that $T(\delta_M(x))=\sum_{n=1}^k a_n\delta_N(y_n)$ and clearly
$$T^*(f_i)(x)=\langle T(\delta_M(x)),f_i \rangle=\sum_{n=1}^k a_n f_i(y_n)\underset{i}{\longrightarrow} \sum_{n=1}^k a_n f(y_n)=T^*(f)(x).$$ Hence $T^*$ is continuous with respect to the topologies of point-wise convergence.

For the converse, assume that there exists $x\in M$ such that $T(\delta_M(x))\notin \span (\delta_N(N))$. For any finite $F\subseteq N$, find by the Hahn-Banach theorem an $f_F\in \Lip_0(N)$ such that $\langle T(\delta_M(x)),f_F\rangle =1$ and $f_F\upharpoonright_{\span\{\delta_N(y): y\in F\}}=0$. Then $(f_F)$, directed by the set inclusion, is a net in $\Lip_0(N)$ that converges point-wise to the zero function in $\Lip_0(N)$. Indeed, for any $y\in N$ and any $F\subseteq N$ finite such that $y\in F$ we have $f_F(y)=\langle \delta_N(y), f_F\rangle = 0$. However, for every finite $F\subseteq N$ we have $T^*(f_F)(x)=\langle \delta_M(x), T^*(f_F)\rangle= \langle T(\delta_M(x)), f_F\rangle =1$ and $T^*(f_F)$ does not converge point-wise to $0$ in $\Lip_0(M)$. So we conclude that if $T^*$ is continuous in the topology of point-wise convergence, then $T(\delta_M(M))\subseteq \span (\delta_N(N))$.
\end{proof}

The following is a direct consequence of Lemma \ref{lemma:span:in:span:iff:point:cont}, Observation \ref{ob:M:free:basis:of:F(M)}  and the description of weak$^\ast$ topology on bounded sets of $\Lipz{M}$.

\begin{proposition}\label{fin:supported:iff:pointwise:homeo}
    Complete metric spaces $M,N$ are \AE-equivalent if and only if there exists an isomorphism between $\Lipz{M}$ and $\Lipz{N}$ that is also a homeomorphism in the topology of point-wise convergence.
\end{proposition}

\begin{remark}\label{rem:graev}
Given a Tychonoff space $M$ with a distinguished base point, denote by $A_g(M)$ the free abelian topological group over $M$ in the sense of Graev (see \cite{Graev} or \cite{AT} for the definition). It is an abelian topological group with the property, that there is a homeomorphic embedding $\delta:M\to A_g(M)$ which preserves the base points and such, that for every continuous map $f$ from $M$ to an abelian topological group $Y$ that preserves the base points, there is a unique continuous homomorphism $\hat{f}:A_g(M)\to Y$ making the following diagram commutative:

\begin{equation*}
     \medskip\hspace{1em}
    \xymatrix{M  \ar@{->}[rd]^f \ar@{->}[d]_\delta & \\
               A_g({M}) \ar@{->}[r]_{\hat{f}} & Y.}
\end{equation*}

During his construction of the free abelian topological group Graev extended in a specific way every continuous pseudometric $d$ on $M$ to a translation-invariant continuous pseudometric on $A_g(M)$. As in \cite{CHDT}, we denote $A_g(M)$ with such a pseudometric $A_d(M)$. If $d$ is a metric on $M$, then $A_d(M)$ becomes a metric group and $\delta:(M,d)\to A_d(M)$ is an isometric embedding. It was noticed by Chasco et al. in \cite[Theorem 2.11]{CHDT} that the metric group generated by $\delta_M(M,d)$ in $\free{M}$ is (isometrically {isomorphic} to) $A_d(M)$ (i.e. there is a group isomorphism of the group generated by $\delta_M(M,d)$ onto $A_d(M)$ which is an isometry).
The metric group $A_d(M)$ uniquely determines\footnote{alternatively, similarly as in \cite{PS}, we may say that the metric group $A_d(M)$ has invariant linear span in the class of normed spaces - that is, whenever $A_d(M)$ isometrically and algebraically embeds into a normed space, then it produces the same linear span (up to a linear isometry).} the normed space of molecules on $M$ constructed by  Arens and Eells in \cite{AE}, and, consequently, its completion -- the Lipschitz-free space $\free{M}$ -- by a standard construction. Namely, using positive homogeneity  of norms, we can extend the distance of each element of $A_d(M)$ from zero to every finite rational linear combination of elements of $M$ because appropriate integer multiple of such a linear combination is an element of $A_d(M)$. Since the rational span of $M$ is dense in $\free{M}$, this distance can be uniquely extended to the whole $\free{M}$. It follows that if for two complete metric spaces $M,N$ there is a group isomorphism of the corresponding metric groups $A_d(M),A_d(N)$ which is at the same time bi-Lipschitz, then $M$ and $N$ are \AE-equivalent. 
\end{remark}

\section{Construction of \texorpdfstring{\AE}--equivalent spaces}
\label{section:examples}
In this section, we present two constructions of \AE-equivalent spaces that provide important examples.

The following fact is a part of \cite[Theorem 3.9 and Proposition 4.2]{AACD2} translated to the language of free bases.

\begin{fact}\label{cor:Ae:eq:not:bilip}
    Let $M$ be a free basis of a Banach space $X$. Define $\mu:M\to X$ by 

    $$
\mu(x) = \left\{
\begin{array}{ll}
0_X &\mbox{for $x =0_X$}
\\
\frac{x}{\norm{x}} & \mbox{for $x\in M\setminus \{0_X\}$}.
\end{array}
\right.$$
If $0_X$ is an isolated point of $M$, then
\begin{enumerate}
    \item  $\mu(M)$ is a free basis of $X$ (in particular, $\mu(M)$ is complete);
    \item  $\mu$ is a homeomorphic embedding.
\end{enumerate}
 In particular, if $0_X$ is an isolated point of $M$, then $M$ and $\mu(M)$ are \AE-equivalent and homeomorphic metric spaces, and $\mu(M)$ is bounded.
\end{fact}

\begin{remark}\label{rem:AE:not:bilip}
  Fact \ref{cor:Ae:eq:not:bilip} and Remark \ref{rem:F-eqv:Free:isom} (i) show that every unbounded complete metric space with an isolated point is \AE-equivalent to a homeomorphic bounded complete metric space. Since bounded and unbounded metric spaces are not bi-Lipschitz isomorphic, this provides a source of examples of \AE-equivalent spaces that are not bi-Lipschitz isomorphic. 
  
\end{remark}

Now we proceed to building examples of \AE-equivalent spaces that are not even homeomorphic.   

Similarly as in \cite{W2}, given two pointed metric spaces $M, N$, their sum $M\coprod N$ is their disjoint union with glued base points, and the metric remains the same on $M$ and $N$ and we define the distance of each $x\in M$ and $y\in N$  as $d(x,y)=d(x,e_M)+d(e_N,y)$. Following \cite{W2},  if $C$ is a closed subset of a complete metric space $M$, then $\sim$ stands for the equivalence relation on $M$ such that $x\sim y$ if $x=y$ or if $x,y\in C$. We define the distance of two equivalence classes $[x],[y]$ as $d([x],[y])=\min(d(x,y),d(x,C)+d(y,C))$. By \cite[Proposition 1.26]{W2} this distance is a complete metric on the set $M/\sim$ of the equivalence classes. We denote this quotient metric space by $M/C$. If $M$ is a pointed metric space, we consider $M/C$ to be a pointed metric space with a base point being the equivalence class which contains the base point of $M$.

The quotient metric spaces can be used to produce \AE-equivalent metric spaces as follows. Fact \ref{fact:quotient} can be extracted from \cite[Theorem 2.13]{AACD_jfa}. We sketch the proof to show the exact relation.
\begin{fact}\label{fact:quotient}
    Let $M$ be a complete metric space and $\varphi:M\to M$ a Lipschitz retraction. Pick a base point $e_M\in \varphi(M)$ arbitrarily. Then $M$ is \AE-equivalent to $\varphi(M)\coprod (M/\varphi(M))$.
\end{fact}

\begin{proof}
Denote $q: M\to M/\varphi(M)$ the natural quotient map. It is not difficult to verify that the mapping $T:\Lipz{\varphi(M)\coprod (M/\varphi(M))}\to\Lipz{M}$ defined by $T(f)=f\circ\varphi+f\circ q$ is an isomorphism with the inverse given for a function $g\in\Lipz{M}$ by
$$T^{-1}(g)(x)=\begin{cases}
g(x)&\textup{ if }x\in\varphi(M),\\
g(y)-g(\varphi(y))&\textup{ if }x=[y]\in M/\varphi(M).
    \end{cases}
$$
Moreover, it is a homeomorphism in the topology of point-wise convergence. We can now apply Proposition \ref{fin:supported:iff:pointwise:homeo} to conclude the \AE-equivalence.
\end{proof}

A counterpart of this construction for free bases of Banach spaces takes the following form. 

\begin{theorem}\label{observation}
    Let $M$ be a free basis of a Banach space $X$ and $\pi:X\to X$ a bounded linear projection. Denote by $\sigma:X\to \ker\pi$ the corresponding projection defined as $\sigma(x)=x-\pi(x)$. 
    If $\pi(M)\subseteq M$, then 
    $\pi(M)\cup\sigma(M)$ is a free basis of $X$.
    In particular, $M$ and $\pi(M)\cup\sigma(M)$ are \AE-equivalent.
\end{theorem}
Before proving the theorem, we show that the set $\pi(M)\cup\sigma(M)$ is closed.

\begin{lemma}\label{lemma:pi(M)cup(sigma(M):closed}
 The sets   $\pi(M)$ and $\sigma(M)$ as in Theorem \ref{observation} are closed in $X$.
\end{lemma}
\begin{proof}
 For this argument, we consider $X$ endowed with the equivalent norm inherited from $\free{M}$. Then the following well-known fact can be derived from \cite{Graev} (see Remark \ref{rem:graev}, for its proof see also \cite{AE, CHDT, W2}): Given $a,b,c,d\in M\subseteq X$, we have 
 \begin{equation}
     \label{eq:norm_molecules}
     \norm{a-b+c-d}=\min\{\norm{a-b}+\norm{c-d}, \norm{a-d}+\norm{c-b}\}.
 \end{equation}
 
The free basis $M$ is closed by definition and so is $\ker\sigma$ because $\sigma$ is continuous. Since $\pi(M)=M\cap\ker\sigma$, it follows that  $\pi(M)$ is closed in $X$. 
  
To show that $\sigma(M)$ is closed in $X$, we pick a Cauchy sequence $(a_n)_{n\in\N}$ of elements of $\sigma(M)$ arbitrarily and find its subsequence which converges to an element of $\sigma(M)$. 
  
  If $\norm{a_n}\to 0$, then $a_n\to 0_X\in\sigma(M)$ and we are done. Otherwise, there is $\varepsilon>0$ and a subsequence $(b_n)_{n\in\N}$ of $(a_n)_{n\in\N}$ such that $\norm{b_n}>\varepsilon \mbox{ for all } n\in\N.$
 For each $n\in\N$ pick $x_n\in M$ such that $b_n=x_n-\pi(x_n)$. Thus we have 
\begin{equation}\label{eq:norm:epsilon}
\norm{b_n}=\norm{x_n-\pi(x_n)}>\varepsilon \mbox{ for all } n\in\N.    
\end{equation}

As $(b_n)_{n\in\N}$ is Cauchy, we may fix $n_0\in\N$ such that $$\norm{b_k-b_l}<\varepsilon \mbox{ for all } k,l>n_0.$$

From \eqref{eq:norm_molecules} and \eqref{eq:norm:epsilon} we get 

\begin{equation}\label{eq:norm:of:the:difference}
    \norm{b_k-b_l}=\norm{x_k-\pi(x_k)-x_l+\pi(x_l)}=\norm{x_k-x_l}+\norm{\pi(x_k)-\pi(x_l)} \mbox{ for all } k,l >n_0.
\end{equation}

Because $(b_n)_{n\in\N}$ is Cauchy, so is $(x_n)_{n\in\N}$ by \eqref{eq:norm:of:the:difference}. Since $M$ is complete, $(x_n)_{n\in\N}$ has limit $x\in M$. Since $\pi$ is bounded, $(\pi(x_n))_{n\in\N}$ has limit $\pi(x)$. Hence $(b_n)_{n\in\N}$ has limit $x-\pi(x)=\sigma(x)\in\sigma(M)$. This finishes the proof.
\end{proof}

Now we can continue to prove the theorem above.
\begin{proof}[Proof of Theorem \ref{observation}]
Pick a Banach space $Y$ and $f\in\Lipz{\pi(M)\cup\sigma(M),Y}$ arbitrarily. 
    Define $h:M\to Y$ by $$h(x)=f(\pi(x))+f(\sigma(x)) \mbox{ for every } x\in M.$$ We claim that $h\in\Lipz{M,Y}$. Indeed, for $x,y\in M$ we have $$\norm{h(x)-h(y)}=\norm{f(\pi(x))-f(\pi(y))+f(\sigma(x))-f(\sigma(y))}\le {L(f)(\norm{\pi}+\norm{\sigma})}\norm{x-y}.$$ Thus $h$ can be  extended to a bounded linear map $\hat{h}$ on $X$. 

   Let us show, that the restriction of $\hat{h}$ to $\pi(M)\cup\sigma(M)$ is $f$.

    Indeed, for $x\in\pi(M)$ we have $\pi(x)=x$ and consequently $$\hat{h}(x)= h(x)=f(\pi(x))+f(\sigma(x))=f(x)+f(\sigma\circ\pi(x))=f(x)+f(0_X)=f(x).$$

Similarly, given $x\in\sigma(M)$, there is $m\in M$ such that $x=\sigma(m)=m-\pi(m)$. Thus $$\hat{h}(x)=\hat{h}(m-\pi(m))=h(m)-h(\pi(m))=f(\pi(m))+f(\sigma(m))-f(\pi\circ\pi(m))-f(\sigma\circ\pi(m))=
$$ $$=f(\pi(m))+ f(x)-f(\pi(m))-f(0_X)=f(x).$$

To summarize: $\hat{h}$ is a bounded linear map which extends $f$, and this extension is unique because $\pi(M)\cup\sigma(M)$ is linearly dense in $X$. 
Since $\pi(M)\cup\sigma(M)$ is a closed subset of $X$ by Lemma \ref{lemma:pi(M)cup(sigma(M):closed}, this  shows that $\pi(M)\cup\sigma(M)$ is a free basis of $X$. 

Obviously, $\span M=\span (\pi(M)\cup\sigma(M))$. Hence $M$ and $\pi(M)\cup\sigma(M)$ are \AE-equivalent.
\end{proof}

Using the quotient spaces, we can find examples demonstrating that \AE-equivalent spaces need not be homeomorphic.

\begin{example}\label{ex:AE:not:homeo}
    There are non-homeomorphic \AE-equivalent spaces.
    Indeed, since $\R$ is a Lipschitz retract of $\R^2$, the spaces $\R^2$ and $\R\coprod \R^2/\R$ are \AE-equivalent by Fact \ref{fact:quotient}.
 \end{example}

    Note that, by {Corollary} \ref{thm:same:dim:for:fin:supported}, two \AE-equivalent spaces must have the same covering dimension. Nevertheless, the next example shows that Assouad dimension of a metric space is not preserved by \AE-equivalence. This dimension was introduced in \cite{Assouad}. For the definition and its properties we refer the reader to \cite{Luuk}, where it is explained, that a metric space is doubling if and only if it has a finite Assouad dimension \cite[Facts 3.3]{Luuk}. We shall also need the trivial facts, that every metric space with infinite bounded uniformly discrete subset has infinite Assouad dimension and that the Assouad dimension of natural numbers (with the usual metric) is equal to 1.

\begin{example}\label{ex:assouad:not:preserved:AE}
    There are \AE-equivalent spaces {such that} one {has} Assouad dimension 1 and the other {has} Assouad dimension $\infty$. In particular, there are two \AE-equivalent spaces one of which has the doubling property and the other one does not.

    Indeed, if $M=\{0,x_1,x_2,x_3,\ldots\}$ is an arbitrary bounded separable uniformly discrete space, then $\free{M}$ is isomorphic to $\free{\N}$ via the isomorphism $T:\free{M}\to\free{\N}$ given by $T(\delta_M(x_n))=\delta_{\N}(n)-\delta_{\N}(n-1)$. So we can use Observation \ref{ob:M:free:basis:of:F(M)} (ii) to conclude that $M$ and $\N$ are \AE-equivalent.

    \end{example}

 \section{\texorpdfstring{\AE}--equivalence preserves the covering dimension}

Given topological spaces $M$ and $N$,we denote the fact that $F$ is a mapping from $M$ to the power set of $N$ 
by $F:M\Rightarrow N$. We say that $F:M\Rightarrow N$ is finite-valued, if $F(x)$ is a finite subset of $N$ for all $x\in M$. Further, we say that $F$ is {\em lower semi-continuous\/}
(abbreviated by {\em \lsc\/}), if for every open $U\subseteq N$ the set
$$\{x\in M:F(x)\cap U\neq\emptyset\}
$$
is open in $M$.

As in \cite[Definition 3.3]{JS}, given a linear basis $M$ of a vector space $X$ we define the finite-valued support function $\supp_M:X\Rightarrow M$ by letting $\supp_M(a)$ be the unique set $K_a\subseteq M$ such that $$a=\sum_{x\in K_a}r_xx \mbox{,  where all $r_x$ are non-zero reals.}$$ 
Here we use the convention that sum over an empty set is the null vector.

The following relatively simple fact is a direct consequence of \cite[Proposition 5.6]{JS}.
\begin{fact}\label{fact:supp:lcs}
    Let $M$ be a linear basis of a vector space $X$ and $\mathcal{T}$ some topology on $X$. Assume that for every $\mathcal{T}\upharpoonright_M$-closed subset $A$ of $M$ and all $x\in M\setminus A$ there is a continuous linear functional $f:(X,\mathcal{T})\to\R$ such that $f(A)\subseteq\{0\}$ and $f(x)\neq 0$. Then $\supp_M:(X,\mathcal{T})\Rightarrow (M,\mathcal{T}\upharpoonright_M)$ is lsc.
\end{fact}

\cite[Corollary 5.4]{Sp1}
provides us with the following fact, where $\dim M$ stands for the {covering} dimension of the metric space $M$.

{\begin{fact}\label{fact:from:HJM}
For $i\in\{0,1\}$ suppose that $M_i$  is a perfectly normal space, and $F_i:M_{1-i}\Rightarrow M_{i}$ is a finite-valued lower semi-continuous mapping. If for every $y\in M_1$
there exists $x\in M_0$ with $x\in F_{0}(y)$ and $y\in F_1(x)${, t}hen $\dim
M_1\le \dim M_0$.    
\end{fact}}

{Combining the above two facts, we can derive the following result about the correspondence of dimensions.}

\begin{theorem}\label{them:main}
    Let $X$ be a vector space endowed with a topology $\mathcal{T}$. For $i\in\{0,1\}$ let $M_i$ be a  linear basis of  $X$ which is perfectly normal with respect to the topology $\mathcal{T}$.  Assume that for every $\mathcal{T}\upharpoonright_{M_i}$-closed subset $A$ of $M_i$ and all $x\in M_i\setminus A$ there is a continuous linear functional $f:(X,\mathcal{T})\to\R$ such that $f(A)\subseteq\{0\}$ and $f(x)\neq 0$. Then $\dim M_0=\dim M_1$.
\end{theorem}
\begin{proof}
    By Fact \ref{fact:supp:lcs}, the support function  $\supp_{M_i}:(X,\mathcal{T})\Rightarrow (M_i,\mathcal{T}\upharpoonright_{M_i})$ is lsc for $i\in\{0,1\}$. Denote by $F_i$ the restrictions of $\supp_{M_i}$ to $M_{1-i}$. It follows from the linear independence of $M_i$ that for all $y\in M_i$ there is $x\in M_{1-i}$ such that $x\in F_{1-i}(y)$ and $y\in F_i(x)$. Hence $\dim M_i\leq\dim M_{1-i}$ for both $i\in\{0,1\}$ by Fact \ref{fact:from:HJM}. 
\end{proof}

{Finally, we show that Theorem \ref{them:main} applies to \AE-equivalent complete metric spaces.}
\begin{corollary}\label{thm:same:dim:for:fin:supported}
If $M$ and $N$ are \AE-equivalent complete metric spaces, then $\dim M=\dim N$. 
\end{corollary}
\begin{proof}
Since the covering dimension is preserved by bi-Lipschitz maps, we may and will assume that $M$ and $N$ are free bases of a common Banach space {$Y$} and $\span M=\span N$. 

Put $X=\span M=\span N$, $M_0=M\setminus\{0_X\}$, $M_1=N\setminus\{0_X\}$ and consider $X$ with the norm topology {inherited from $Y$}. We claim that the conditions of Theorem \ref{them:main} are satisfied. Indeed, for both $i\in\{0,1\}$, $M_i$ is linearly independent by Remark \ref{rem:F-eqv:Free:isom} (v) and thus it is a linear basis of $X=\span M_i$. Moreover, given a closed subset $A\subset M_i$ and $x\in M_i\setminus A${,} there is a function $f\in\Lipz{M_i\cup\{0_X\}}$ such that $f(A)\subseteq \{0\}$ and $f(x)=1$. Since $M_i\cup \{0_X\}$ is a free basis, such a function can be extended to a bounded linear functional on $X$ {and} the equality $\dim M_0=\dim M_1$ follows. 

To complete the proof it suffices to observe that given a point $x$ of a perfectly normal space $M$ (of a metric space in particular), we have $\dim M=\dim M\setminus\{x\}$. This folklore fact follows from the Countable Sum Theorem \cite[3.1.8]{Eng}  and from the Subspace Theorem for the \v{C}ech-Lebesgue dimension of perfectly normal spaces \cite[2.1.3 and 3.1.19]{Eng}{,}  and from the fact that open subsets of a perfectly normal space are $F_\sigma$-sets. 
\end{proof}

\begin{remark}
Let us note, that for bounded separable metric spaces{,} {Corollary} \ref{thm:same:dim:for:fin:supported} can be proved by a technique developed in \cite{Gulko}. Indeed, for \AE-equivalent spaces $M,N$,  by Proposition \ref{fin:supported:iff:pointwise:homeo}, the spaces $\Lipz{M}$ and $\Lipz{N}$ are linearly homeomorphic (hence uniformly homeomorphic) in the topology of point-wise convergence. Consequently, the spaces $\Lip(M)$ and $\Lip(N)$ of all Lipschitz functions on $M$ and $N$ respectively are uniformly homeomorphic in the topology of point-wise convergence. Since the space of Lipschitz functions on a metric space is closed under point-wise addition and scalar multiplication and for bounded metric space also under multiplication   (see for instance \cite[Propositions 1.29 and 1.30]{W2}), and Lipschitz functions separate points and closed sets, $\Lip(M)$ and $\Lip(N)$ form so called QS-algebras (see \cite[1.1 Definition]{Gulko}). To derive {Corollary} \ref{thm:same:dim:for:fin:supported} for bounded separable metric spaces {it} suffices now to apply \cite[1.4. Proposition]{Gulko}.
\end{remark}

\end{document}